\documentclass[11pt,twoside]{amsart}
\usepackage{hyperref, tikz, fullpage}
\usepackage{placeins, enumerate,longtable}
\usepackage{amsthm, amssymb,amsmath,latexsym,amscd, mathrsfs}
\usepackage{mathtools}
\usepackage[utf8]{inputenc}
\usepackage[T1]{fontenc}
\usepackage{enumitem,color}
        
\usepackage{caption}
\usepackage{lscape}
\newtheorem{theorem}{Theorem}[section]
\newtheorem{proposition}[theorem]{Proposition}

\theoremstyle{definition}
\newtheorem{definition}[theorem]{Definition}
\newtheorem{example}[theorem]{Example}

\numberwithin{equation}{section}
\newcommand{\Cal}[1]{\mathcal{#1}}

\newcommand{\Fq}{\mathbf{F}_q}

\newcommand{\Z}{\mathbb{Z}}
\newcommand{\R}{\mathbb{R}}

\newcommand{\rep}{\mathrm{Rep}}
\begin{document}
\title{Commuting probability and simultaneous conjugacy classes of commuting tuples in a group}
\author{Uday Bhaskar Sharma}
\email{udaybsharmaster@gmail.com}
\address{Tata Institute of Fundamental Research, Dr. Homi Bhabha Road, Navy Nagar, Colaba, Mumbai - 400005, India}
\author{Anupam Singh}
\email{anupamk18@gmail.com}
\address{Indian Institute of Science Education and Research (IISER) Pune,  Dr. Homi Bhabha Road Pashan, Pune 411008, India}
\today
\thanks{The second named author acknowledges support of the SERB Core Research Grant CRG/2019/000271 for this work.}

\begin{abstract}
Let $G$ be a finite group. This expository article explores the subject of commuting probability in the group $G$ and its relation with simultaneous conjugacy classes of commuting tuples in $G$. We also point out the relevance of this topic in context of topology where similar problems are studied when $G$ is a Lie group. 
\end{abstract}
\maketitle
\section{Introduction}

Let $G$ be a finite group. If $G$ is not commutative then there are pairs of elements in $G$ which do not commute. The property of non-commutativity of elements of $G$ can be measured by commuting probability, which is defined by 
$$ cp_2(G) = \frac{ |\{(x,y)\in G\times G \mid x y = y x \}| }{|G|^2} $$
and more generally by $cp_n(G)$ (see Section~\ref{CP} for definition). Erd\"os and Turan (see~\cite{ET}) took probabilistic approach to group theory in a series of papers titled ``Statistical Group Theory''. Among several results they proved that $cp_2(G)=\frac{k(G)}{|G|}$ where $k(G)$ is the number of conjugacy classes in $G$. The commuting probabilities (under different names) have been studied for more general groups, for example, compact groups (see~\cite{AG,HR1, HR2,RE}). Following the work of Feit and Fine~\cite{FF} for general linear groups, this has been further investigated for varieties (see for example, see~\cite{BM}). There are various generalisations of commuting probability for groups and we refer to the survey article~\cite{DNP} for the same. Fulman and Guralnick studied this problem for Lie algebras over a finite field in~\cite{FG}. In~\cite{GR}, Guralnick and Robinson investigated how far this quantity can help in classifying finite groups. We present this in more detail in Section~\ref{CP}. In \cite{Sh}, Shalev looked at some more generalisations and suggested several research problems. In \cite{Ru}, Russo generalized the problem of commuting probabilities to relative commutativity degree of two subgroups of a group.

For finite groups, the computation of commuting probability can be done using the branching matrix $B_G$ for the group $G$. The branching matrix keeps information about how conjugacy classes of various centraliser subgroups of $G$ are obtained using the conjugacy classes of $G$ and this process is repeated iteratively by replacing $G$ with a centraliser subgroup of $G$. In~\cite{Sh1,Sh2,SS} this is computed for some small rank classical groups. In Section~\ref{BR} we explore this relation and compute the probabilities for $GL_2(\Fq), GL_3(\Fq), U_2(\Fq), U_3(\Fq)$ and $Sp_2(\Fq)$ using the branching matrix. In particular, the problem of computing $cp_n(G)$ is reduced to computing the $(n-1)$th power of the matrix $B_G$ (see Theorem~\ref{branchingmatrix}). We prove that the commuting probabilities $cp_n(U_2(\Fq))$ and $cp_n(GL_2(\Fq))$ are same (see Proposition~\ref{cp-gl-u}) for all $n$.

Witten (see~\cite{Wi1,Wi2}), while studying supersymmetry, required understanding of $Hom(\mathbb Z^n, G)$ for $n=2,3$ for a Lie group $G$. This started the study of topological properties of $Hom(\mathbb Z^n, G)$ and $Rep(\mathbb Z^n,G)$. We present this briefly in Section~\ref{CTT}. 


\section{Commuting probability in groups}\label{CP}
Let $G$ be a finite group, and $n$ a positive integer. Let $G^{(n)}$ denote the set of $n$-tuples in $G$ which commute pairwise, i.e.,  
$$G^{(n)} = \{(g_1,\ldots,g_n)\in G^n \mid  g_ig_j = g_jg_i, \forall~ 1\leq i \neq j \leq n \}.$$
We denote the probability of finding a $n$-tuple (or $n$ elements in $G$ randomly) of elements of $G$, which commute pair-wise by 
$$cp_n(G) = \displaystyle\frac{|G^{(n)}|}{|G|^n}.$$
In particular, $cp_2(G)$ denotes the probability of finding a commuting pair in $G$. In literature, this quantity has different notation, for example, Gustafson (see~\cite{Gu}) denotes $cp_2(G)$ by $Pr(G)$, and  Lescot (see~\cite{Le}) denotes $cp_n(G)$ by $d_{n-1}(G)$ and $cp_2(G)$ by simply $d(G)$. The topic of the probability of finding a commuting $n$-tuple of elements of a group $G$ (including Lie groups) has been of considerable interest. As mentioned in the introduction, this subject started with a series of papers by Erd\"os and Turan on the subject of ``Statistical Group Theory'' and a paper by Gustafson. In~\cite[Theorem IV]{ET} (see also~\cite{Gu}), using a simple argument, it is proved that for a finite group $G$, 
\begin{equation}\label{cp2}
cp_2(G) = \displaystyle\frac{k(G)}{|G|}
\end{equation}
where $k(G)$ is the number of conjugacy classes in $G$.
For a non-abelian finite group $G$ it is easy to observe (see~\cite[Section 1]{Gu}) that 
$$cp_2(G) \leq \displaystyle\frac{5}{8}.$$ 
We mention some examples to illustrate this concept.
\begin{example}
For the quaternion group $Q_8$ and the dihedral group $D_4$ of order $8$ the number of conjugacy classes is $5$. Thus, $cp_2(Q_8)=cp_2(D_4)=\frac{5}{8}$. This shows that equality holds in the above equation for $cp_2(G)$, when $G$ is any non-abelian group of order $8$. 
\end{example}
\noindent Dixon, in 1970, proposed and later in 1973 proved, that for a finite non-abelian simple group $G$, the commuting probability $cp_2(G) \leq \displaystyle\frac{1}{12}$. 
\begin{example}
The alternating group $A_5$ has $5$ conjugacy classes and thus $cp_2(A_5)=\frac{5}{60}=\frac{1}{12}$. Thus, the equality holds for the group $A_5$.
\end{example}
\begin{example}
For the group $PSL_2(q)$ when $q$ is odd, the number of conjugacy classes $k(PSL_2(q))=\frac{(q+5)}{2}$. We get, $cp_2(PSL_2(q))=\frac{(q+5)}{(q+1)q(q-1)}$. The group $PSL_2(3)\cong A_4$ is not simple and $cp_2(PSL_2(3))=\frac{1}{3}$. For $q\geq 5$, a simple computation gives us $cp_2(PSL_2(q))\leq \frac{1}{12}$ with equality for $q=5$.
\end{example}

In~\cite{Le}, Lescot studied $cp_2$ up to isoclinism. Let us recall the definition of isoclinic groups. 
\begin{definition}[Isoclinic Groups]
Two groups $G$ and $H$ are said to be isoclinic, if there is an isomorphism $\phi \colon G/Z(G) \rightarrow H/Z(H)$, and an isomorphism $\psi \colon [G,G] \rightarrow [H,H]$, such that the following diagram commutes:
\[ \begin{CD}
G/Z(G) \times G/Z(G) @>\phi \times \phi  >> H/Z(H)\times H/Z(H)\\
 @ VVV                              @ VVV \\
[G,G] @>\psi >>  [H,H]
\end{CD} \]
where the vertical maps are the commutator maps.
\end{definition}
\noindent Lescot proved that~\cite[Lemma 2.4]{Le} if $G$ and $H$ are isoclinic, then $cp_2(G) = cp_2(H)$. Further, he proved that if $cp_2(G) \geq \frac{1}{2}$,  then (see \cite[Theorem 3.1]{Le}) $G$ is isoclinic to either $\{1\}$, an extra-special $2$-group, or to the symmetric group $S_3$. Guralnick and Robinson proved that (see~\cite[Theorem 11]{GR}) for a finite group $G$ if $cp_2(G) > \frac{3}{40}$ then either $G$ is solvable, or $G \equiv A_5\times T$, where $T$ is some abelian group. In~\cite{GR}, Guralnick and Robinson revisited this topic for finite groups, and proved several interesting results. They proved, for a finite solvable group $G$ with derived length $d \geq 4$, $cp_2(G) \leq \displaystyle\frac{4d-7}{2^{d+1}}$, and for a $p$-group $G$, with derived length $d \geq 2$, $cp_2(G) \leq \displaystyle\frac{p^d+p^{d-1} - 1}{p^{2d-1}}$. Lescot, along with Hung Ngoc Nguyen and Yong Yang~\cite{LNY} determined the supersolvability of finite groups, for groups $G$ with $cp_2(G)$ above some certain values. In \cite{Eb}, Eberhard proved a structure theorem for the set $\{cp_2(G)\mid G = \mathrm{finite~group}\} $ of values of commuting probabilities of all finite groups. In~\cite{CA} it is classified that a group having the property that elements of same order are conjugate then group is abelian.

Feit and Fine (see~\cite{FF}) computed the number of ordered pairs of commuting matrices in $GL_d(\mathbb F_q)$ which is 
$$P(d)=|GL_d(q)^{(2)}| = q^{d^2}f(d) \sum_{\pi(d)} \frac{q^{k(\pi)}}{f(b_1)f(b_2)\cdots f(b_d)}$$
where $f(t)= \displaystyle \prod_{i=1}^t \left(1-\frac{1}{q^i}\right)$ when $t\geq 1$ and $f(0)=1$, and $\pi(d)=1^{b_1}2^{b_2}\ldots d^{b_d}$ is a partition of $d$ written in power notation, and $k(\pi) = b_1 + \cdots + b_d$. A generating function for this is 
$$\sum_{d\geq 0} \frac{P(d) u^d}{|GL_d(q)|} = \prod_{i\geq 1} \prod_{l\geq 0} \frac{1}{\left( 1-\frac{u^i}{q^{l-1}}\right)}.$$
Fulman and Guralnick (in~\cite{FG}) have further computed this for the Lie algebra of finite unitary and symplectic groups, and have found generating functions.  

Let us now move on to describe what is known for more general $cp_n(G)$ when $n\geq 3$. 
Lescot (see~\cite[Lemma 4.1]{Le}) gave a recurrence formula:
\begin{equation}\label{lescot}
cp_n(G) = \frac{1}{|G|} \sum_{i = 1}^{k(G)}  \frac{cp_{n-1}\left (\mathcal Z_G(g_i)\right)}{|\mathscr{C}(g_i)|^{n-2}},
\end{equation}
where $\mathscr{C}(g_i)$ denotes the conjugacy class of $g_i$ in $G$, and $\mathcal Z_G(g_i)$ denotes the centralizer of $g_i$ in $G$. We use this formula to compute the commuting probability for the group $GL_2(\Fq)$.
\begin{example}\label{cpGL}
 For $G = GL_2(\Fq)$, the probabilities $cp_n(GL_2(\Fq))$ for $n$ up to $5$, are as follows,
\begin{center}
\begin{tabular}{|c|c|c|c|}\hline
$n$ & $cp_n(GL_2(\Fq))$ & $n$ & $cp_n(GL_2(\Fq))$ \\ \hline
&&& \\
$2$ & $\frac{1}{q^{2} - q}$  & $3$ & $\frac{q^{2} + q + 2}{q^{6} - 2 q^{4} + q^{2}}$\\
&&&\\
$4$ & $\frac{q^{3} + q^{2} + 4 q + 1}{q^{9} - 3 q^{7} + 3 q^{5} - q^{3}} $ & 
$5$ & $\frac{q^{4} + q^{3} + 7 q^{2} + q + 2}{q^{12} - 4 q^{10} + 6 q^{8} - 4 q^{6} + q^{4}} $\\ 
&&& \\ \hline
\end{tabular}
\end{center}
The data required for this computation is conjugacy classes and centralisers of $GL_2(\Fq)$ which can be easily computed.
\end{example}
In Section~\ref{BR}, we give a formula to compute these probabilities using, what we call, the branching matrix of the group. Before going further we briefly describe that the notion of commuting probability makes sense for a compact group as well.

\subsection{Compact groups}

Let $G$ be a compact group with Haar measure $\mu$. Equip $G^n$ with the product measure $\mu^n$. Since the multiplication map $G^n\rightarrow G$ is continuous, the set $G^{(n)}$ is measurable. Define the commuting probability 
as follows: 
$$cp_n(G) = \displaystyle\frac{\mu^n(G^{(n)})}{\mu^n(G^n)}.$$
If we consider normalized Haar measure the formula would be simply $cp_n(G) = \mu^n(G^{(n)})$.
Gustafson (see~\cite{Gu} Section 2 Theorem) proved that for a non-abelian compact group $G$, the commuting probability $cp_2(G) \leq \displaystyle\frac{5}{8}$. In~\cite{RE, ER} several properties of $cp_n(G)$ (which is denoted as $d_{n+1}(G)$) are studied. In~\cite[Theorem 3.8]{RE} it is proved that $cp_n(G)$ takes same value on isoclinic compact groups, further, for a non-abelian compact group $G$,
$$cp_{n}(G) \leq \frac{3.2^{n-1}-1}{2^{2n-1}}$$
with equality for finite group of order $8$. 
Here are some more results from~\cite{ER}. Let $G$ be a non-abelian compact group, such that $G/Z(G)$ is a $p$-elementary abelian group of rank $2$ (where $p$ is a prime), then $cp_n(G) = \displaystyle\frac{p^n+p^{n-1} -1}{p^{2n-1}}$. Further, they proved that if $G$ is a non-abelian compact group and $G/Z(G)$ is a $p$-group ($p$, prime) satisfying $cp_n(G) =\displaystyle\frac{p^n+p^{n-1} -1}{p^{2n-1}}$, then $G/Z(G)$ is a $p$-elementary abelian group of rank $2$. They also proved that if $[G:Z(G)] = p^k$ then 
$$cp_n(G) = \frac{\left(p-1\right)\left(\displaystyle \sum_{i=0}^{n-2}p^{i(k-1)} + p^{(n-1)k - n + 2}\right)}{p^{(n-1)k + 1}}.$$
Hofmann and Russo (see~\cite{HR1, HR2}) attempted to get the structure of compact groups given certain conditions on $cp_n(G)$. In particular, they have proved the following result (see~\cite[Theorem 1.2]{HR1}). Let $G$ be a compact group with Haar measure $\mu$ and FC-center $F$. Then, $cp_2(G)>0$ if and only if $F$ is open in $G$, which is, if and only if the center $\mathcal Z(F)$ of $F$ is open in $G$. Recall, the FC-center of a group is the set of elements whose conjugacy class is finite.

\section{Simultaneous conjugacy and branching rules}\label{BR}

For a finite group $G$, we can analyse the set $G^{(n)}$ (defined in Section~\ref{CP}) under the simultaneous conjugation action on it by $G$, that is, we have 
\begin{eqnarray*} G\times G^{(n)} &\rightarrow& G^{(n)} \ \ \ \ \ \ {\rm given\ by}\\ 
 g.(x_1,\ldots, x_n) & = & (gx_1g^{-1}, \ldots, gx_ng^{-1}).
 \end{eqnarray*}
Let $c_G(n)$ denote the number of orbits under this action. These orbits are also called simultaneous conjugacy classes in $G^{(n)}$. The main problem here is to determine $c_G(n)$. The above question can be also asked for a finite dimensional algebra $\Cal{A}$ over the finite field $\Fq$ under the action of  $\Cal{A}^* $, the group of units of $\Cal{A}$. The number of orbits is denoted as $c_{\Cal A}(n)$ in this case. The number $c_G(n)$ is determined for some classical groups of small rank by the authors in~\cite{Sh1,SS}. When $n=2$, 
$$G^{(2)} = \displaystyle \bigcup_{x\in G} \{x\}\times \mathcal Z_G(x),$$ thus, to get size of this set we need to look at the elements which are in same {\em $z$-classes} also called {\em types}. Two elements are said to be in same $z$ class if their centralisers are conjugate. The computation for $G^{(n)}$ can be reduced to the centraliser subgroups up to conjugacy. This usually simplifies the computation as one has to deal with much fewer centralizer subgroups. In \cite{Sh1,SS}, $c_G(n)$ is computed for $GL_2(\Fq)$, $GL_3(\Fq)$, $U_2(\Fq)$, $U_3(\Fq)$, $Sp_2(\Fq)$ and $Sp_4(\Fq)$. 

Using the data on $z$-classes, and the centralizers within representatives of the $z$-classes, the enumeration of simultaneous conjugacy of commuting elements is reduced to determining what is called the branching table/matrix $B_G$ of the group $G$. Before going any further, let us understand what branching means. Let $\mathscr C(x)$ be a conjugacy class of $x$ in $G$ and $\mathcal Z_G(x)$ be its centralizer. The conjugacy classes of $\mathcal Z_G(x)$ are called branches of the conjugacy class $\mathscr C(x)$. Now, since $\mathcal Z_G(x)$ is a proper subgroup of $G$, some of its conjugacy classes could be a class of $G$ and some are not (for example the ones which are obtained by splitting a class of $G$). If $\mathcal Z_G(x)$ has a conjugacy class which is not a class of $G$, it's called a new one. This process is repeated iteratively for centralizer subgroups within centralizer subgroups till we do not get a new one. Notice that we need to store data for an element which is representative of its $z$-class. The branching table has its rows and columns indexed by the types of the conjugacy classes of $G$ as well as the new ones (that is $z$-classes in $G$ as well as the same within centralizers iteratively).  For finite classical groups, it leads to some interesting combinatorics.

In the case of finite groups, computation of the branching matrix $B_G$ is helpful in explicitly determining $cp_n(G)$ for $n \geq 2$. In particular, Lescot's formula given in Equation~\ref{lescot} can be recasted using simultaneous conjugacy classes of commuting elements. The following theorem is proved in~\cite[Theorem 1.1]{SS}. 
\begin{theorem}\label{branchingmatrix}
Let $G$ be a finite group and $n\geq 2$, an integer. The probability that a $n$-tuple of elements of $G$ commute is 
$$cp_n(G) = \frac{c_G(n-1)}{|G|^{n-1}}=  \frac{\mathbf{1}B_G^{n-1}e_1}{|G|^{n-1}}$$
where $\mathbf{1}$ is a row matrix, with all $1$'s, and $e_1$ is a column matrix with first entry $1$, and $0$ elsewhere.
\end{theorem}
\noindent We explain how to get the first formula for $n=2$ and $3$. Let $k$ be the number of conjugacy classes in $G$. Let $g_1, \ldots, g_k$, denote the representatives of the conjugacy classes in $G$. The following iterative formula of $c_G(n)$ is immediate. We set $c_G(0)=1$, thus for $n\geq 1$ we have,
\begin{equation}\label{EComm1}
c_G(n) = \sum_{i=1}^k c_{\mathcal Z_G(g_i)}(n-1).
\end{equation}
Now, for $n = 2$, we know from Equation~\ref{cp2} that $cp_2(G) = \frac{k}{|G|}$, which is equal to $\frac{c_G(1)}{|G|}$. We demonstrate this for $n=3$ using Lescot's formula given in Equation~\ref{lescot}. 
\begin{eqnarray*}
cp_3(G)&=& \frac{1}{|G|}\sum_{i=1}^k \frac{cp_2(\mathcal Z_G(g_i)) }{|\mathscr{C}(g_i)|} = \frac{1}{|G|}\sum_{i=1}^k \frac{1}{|\mathscr{C}(g_i)|}\frac{c_{\mathcal Z_G(g_i)}(1)}{|\mathcal Z_G(g_i)|}\\
&=& \frac{1}{|G|}\sum_{i=1}^k \frac{c_{\mathcal Z_G(g_i)}(1)}{|G|} = \frac{c_G(2)}{|G|^2}~\text{~from Equation~\ref{EComm1}}.
\end{eqnarray*}
In fact, the proof of the first equality of the above theorem is along the similar lines which is  given in~\cite{SS}, Section 7. The second equality is proved in~\cite{SS} Lemma 7.1. Below we give some examples of $cp_n(G)$ for some small rank classical groups where $B_G$ is already known.

In~\cite{Gu2}, Guralnick proved that for a field $F$ and two commuting matrices $A,B \in M_d(F)$, the algebra $\Cal{A}(A,B)$ generated by $A$ and $B$, is of dimension $\dim(\Cal{A})\leq d$. In general he looked at the dimension of the space $C(m,d)$, a commutative subalgebra of $M_d(F)$ generated by $m$ elements, and its irreducibility (or non-irreducibility) and dimension etc. In~\cite{Sh1}, the first named author studied the set $\Cal{A}^{(n)}$ of $n$-tuples of commuting elements of $\Cal{A}$ for $\Cal{A} = M_d(\Fq)$, for $d = 2,3,4 $. In this paper, it was proved (\cite[Theorem~1.5]{Sh1}) that $c_{M_d(\Fq)}(n)$ (denoted $c_{d,n}(q)$ in the paper) is a polynomial in $q$ with non-negative integer coefficients for $d = 2,3,4$ and $n \geq 1$. This led to some questions about generalising the mentioned result for $d \geq 5$, i.e., to check if $c_{M_d(\Fq)}(n)$ is a polynomial with non-negative integer coefficients for $d \geq 5$. 

In~\cite{SS}, we looked at the unitary groups $U_d(\Fq)$ for $d = 2,3$, and the symplectic groups $Sp_{2l}(\Fq)$ for $l = 1, 2$. We recall the branching matrix and compute the commuting probabilities using SAGEmath~\cite{Sagemath} up to $n\leq 5$. We also observed an interesting duality between the formulas of $c_{U_2(\Fq)}(n)$ and $c_{GL_2(\Fq)(n)}$, using the branching tables, which leads to proving that commuting probability for these two groups are same. 

\begin{example}
For $G = U_2(\Fq)$, and $n \geq 2$, the branching matrix is (see~\cite[Theorem 1.2]{SS})
$$B_{U_2(\Fq)} = \left(\begin{array}{rrrr}
q + 1 & 0 & 0 & 0 \\
q + 1 & q(q+1) & 0 & 0 \\
\frac{1}{2} (q+1)q & 0 & (q+1)^2 & 0 \\
\frac{1}{2} (q^2-q-2) & 0 & 0 & q^{2} - 1
\end{array}\right)$$
and the probabilities for $n\leq  5$ are mentioned in the table below,
\begin{center}
\begin{tabular}{|c|c|c|c|}\hline
$n$ & $cp_n(U_2(\Fq))$ & $n$ & $cp_n(U_2(\Fq))$ \\ \hline &&&\\
$2$ & $\frac{1}{q^{2} - q}$ & $3$ & $\frac{q^{2} + q + 2}{q^{6} - 2 q^{4} + q^{2}}$\\
&&&\\
$4$ & $\frac{q^{3} + q^{2} + 4 q + 1}{q^{9} - 3 q^{7} + 3 q^{5} - q^{3}} $ & 
$5$ & $\frac{q^{4} + q^{3} + 7 q^{2} + q + 2}{q^{12} - 4 q^{10} + 6 q^{8} - 4 q^{6} + q^{4}} $\\ &&& \\ \hline
\end{tabular}
\end{center}
Notice that for $G = GL_2(\Fq)$ (see Example~\ref{cpGL}) and $G = U_2(\Fq)$ when $n\leq 5$, the probabilities $cp_n(G)$ are the same even though the branching matrices are not same. In fact, this is true for all $n$. 
\end{example}
\begin{proposition}\label{cp-gl-u}
For $n\geq 2$, $cp_n(GL_2(\Fq)) = cp_n(U_2(\Fq))$. 
\end{proposition}
\begin{proof}
First we note that $|U_2(\Fq)| = q(q+1)(q^2-1) = q(q+1)^2(q-1)$, and $|GL_2(\Fq)| = q(q-1)(q^2-1)= q(q-1)^2(q+1)$, and, there is a duality between the formulas of $c_{GL_2(\Fq)}(n)$ and $c_{U_2(\Fq)}(n)$ for $n\geq 1$ (see~\cite[Corollary 1.4(1)]{SS}). 
So, for $n \geq 2$, we have the probability, 
\begin{center}
$
\begin{array}{l}
cp_{n+1}(GL_2(\Fq)) = 
\frac{1}{q^n(q-1)^{2n}(q+1)^n} \left(\begin{matrix}
  \displaystyle\sum_{a+b+c+d = n}(q-1)^{a+2b+c+d}q^c(q+1)^d \\
 - 2q(q-1)\left(\displaystyle\sum_{a+b+c+d = n-1} (q-1)^{a+2b+c+d} q^c(q+1)^d\right)\\
~ + q^2(q-1)^2\left(\displaystyle\sum_{a+b+c+d = n-2}(q-1)^{a+2b+c+d} q^c(q+1)^d\right)  \\
~- (q-1)^3 \left(\displaystyle\sum_{a+b+c+d = n-3}(q-1)^{a+2b+c+d} q^c(q+1)^d \right)
\end{matrix}\right).
\end{array}
$
\end{center}
After doing the cancellations in the above equation, we get:
\begin{eqnarray*}
cp_{n+1}(GL_2(\Fq)) &=&  \sum_{a+b+c+d = n}(q-1)^{b-n}q^{c-n}(q+1)^{d-n} \\ &-& 2\left(\sum_{a+b+c+d = n-1} (q-1)^{b-n} q^{c+1-n}(q+1)^{d-n}\right) \\
& +& \left(\sum_{a+b+c+d = n-2}(q-1)^{b-n} q^{c+2-n}(q+1)^{d-n}\right) \\ &-& \left(\sum_{a+b+c+d = n-3}(q-1)^{b-n} q^{c-n}(q+1)^{d-n} \right).
\end{eqnarray*}
In the case of $U_2(\Fq)$, we have:
$$
 \begin{array}{l}
 cp_{n+1}(U_2(\Fq)) =  \frac{1}{q^n(q+1)^{2n}(q-1)^n} \left(\begin{matrix} \displaystyle\sum_{a+b+c+d = n} (q+1)^{a+2d+b+c}(q-1)^b q^c \\
 ~ - 2q(q+1)\left(\displaystyle\sum_{a+b+c+d = n-1}(q+1)^{a+2d+b+c}(q-1)^b q^c\right)\\
 ~ + q^2(q+1)^2\left(\displaystyle\sum_{a+b+c+d = n-2}(q+1)^{a+2d+b+c}(q-1)^b q^c\right)  \\
 ~- (q+1)^3 \left(\displaystyle\sum_{a+b+c+d = n-3}(q+1)^{a+2d+b+c}(q-1)^b q^c \right)
 \end{matrix}\right)\\
 \end{array}
 $$
 After doing the cancellations in the above equation, we see that for $n \geq 2$, $cp_n(GL_2(\Fq)) = cp_n(U_2(\Fq))$.
\end{proof}

\begin{example}
For $G = Sp_2(\Fq)$ the branching matrix is (see~\cite[Theorem 1.3]{SS})
$$B_{Sp_2(\Fq)} = \left(\begin{array}{rrrrr}
2 & 0 & 0 & 0 &0\\
2 & 2q & 0 & 0 & 0 \\
2 & 0 & 2q & 0 &0 \\
\frac{1}{2}(q-3) & 0 & 0 & q- 1&0\\
\frac{1}{2}(q-1) & 0 & 0 & 0&q+1
\end{array}\right)$$
and the commuting probability is,
\begin{center}
\begin{tabular}{|c|c|c|c|}\hline
$n$ & $cp_n(Sp_2(\Fq))$ & $n$ & $cp_n(Sp_2(\Fq))$ \\ \hline
&&&\\
$2$ & $\frac{q + 4}{q^{3} - q}$ & $3$ & $\frac{q^{2} + 8 q + 9}{q^{5} - 2 q^{3} + q}$\\
&&&\\
$4$ & $\frac{q^{3} + 16 q^{2} + 19 q + 16}{q^{7} - 3 q^{5} + 3 q^{3} - q} $ & $5$ & $\frac{q^{4} + 32 q^{3} + 38 q^{2} + 32 q + 33}{q^{9} - 4 q^{7} + 6 q^{5} - 4 q^{3} + q}$.\\
&&&\\  \hline
\end{tabular}
\end{center}
\end{example}

\begin{example}
For $G = GL_3(\Fq)$ the branching matrix is,
$$B_{GL_3(\Fq)} = \left(\begin{smallmatrix}
q - 1 & 0 & 0 & 0 & 0 & 0 & 0 & 0 \\
q - 1 & q(q-1) & 0 & 0 & 0 & 0 & 0 & 0 \\
(q-1)(q-2) & 0 & (q-1)^2 & 0 & 0 & 0 & 0 & 0 \\
q - 1 & q^{2} - 1 & 0 & q^2(q-1) & 0 & 0 & 0 & 0 \\
(q-1)(q-2)& (q-1)(q-2)q & (q-1)^2 & 0 & q(q-1)^2 & 0 & 0 & 0 \\
{q-1 \choose 3} & 0 & (q-1){q-1 \choose 2} & 0 & 0 & (q-1)^3 & 0 & 0 \\
(q-1){q \choose 2}& 0 & (q-1){q \choose 2} & 0 & 0 & 0 & (q-1)^2(q+1) & 0 \\
\frac{(q^3-q)}{3} & 0 & 0 & 0 & 0 & 0 & 0 & q^{3} - 1
\end{smallmatrix}\right)$$
and the values of $cp_n(G)$ are,

\begin{center}
\begin{tabular}{|c|c|c|c|}\hline
$n$ & $cp_n(GL_3(\Fq))$ & $n$ & $cp_n(GL_3(\Fq))$\\ \hline
&&&\\
$2$ & $\frac{1}{(q - 1)^{2} q^{2} (q^{2} + q + 1)}$ & $3$ & $\frac{q^{4} + q^{3} + q^{2} + 4}{(q + 1)^{2}  (q - 1)^{4}  q^{6}  (q^{2} + q + 1)^{2}}$\\
&&&\\
$4$ & $\frac{q^{6} + q^{5} + 2 q^{4} + q^{3} + 8 q^{2} + 4 q + 1}{(q + 1)^{3} (q - 1)^{6}  q^{9}(q^{2} + q + 1)^{3}} $ & 
$5$ & $\frac{q^{8} + q^{7} + 4 q^{6} + 23 q^{4} - 2 q^{3} + 13 q^{2} - q + 4}{(q + 1)^{4}  (q - 1)^{8} q^{12}  (q^{2} + q + 1)^{4}} $.\\ 
&&&\\ \hline
\end{tabular}
\end{center}
\end{example}

\begin{example}
For $G = U_3(\Fq)$, the branching matrix is,
$$B_{U_3(\Fq)} = \left(\begin{smallmatrix}
q+1 & 0 & 0 & 0 & 0 & 0 & 0 & 0 \\
q+1 & q(q+1) & 0 & 0 & 0 & 0 & 0 & 0 \\
q(q+1) & 0 & (q+1)^2 & 0 & 0 & 0 & 0 & 0\\
q+1 & q^2 - 1 & 0 & (q+1)q^2 & 0 & 0 & 0 & 0 \\
q(q+1) & (q+1)q^2 & (q+1)^2 & 0 & q^2(q+1) & 0 & 0 & 0 \\
{q+1}\choose 3 & 0 &(q+1){{q+1}\choose 2} & 0 & 0 & (q+1)^3 & 0 & 0 \\
\frac{(q+1)(q^2-q-2)}{2} & 0 & \frac{(q+1)(q^2-q-2)}{2}& 0 & 0 & 0 & (q+1)(q^2-1)& 0\\
\frac{q^3-q}{3} & 0 & 0 & 0& 0& 0& 0 &  q^3 + 1 \end{smallmatrix}\right),$$ and the probabilities for up to $k = 5$ are mentioned in the table below,

\begin{center}
\begin{tabular}{|c|c|c|c|}\hline
$n$ & $cp_n(U_3(\Fq))$ & $n$ & $cp_n(U_3(\Fq))$ \\ \hline
&&&\\
$2$ & $\frac{q^{2} + q + 2}{(q - 1) (q + 1)^{2}q^{3} (q^{2} - q + 1)}$ & $3$ & $\frac{q^{4} + q^{3} + 5 q^{2} + 4 q + 2}{(q - 1)^{2} (q + 1)^{4} q^{6}  (q^{2} - q + 1)^{2}}$\\
&&&\\
$4$ & $\frac{q^{6} + q^{5} + 8 q^{4} + 9 q^{3} + 14 q^{2} + 4 q + 1}{(q - 1)^{3}  (q + 1)^{6}  q^{9}  (q^{2} - q + 1)^{3}} $ & $5$ & $\frac{q^{8} + q^{7} + 12 q^{6} + 16 q^{5} + 37 q^{4} + 20 q^{3} + 17 q^{2} + 5 q + 2}{(q - 1)^{4}  (q + 1)^{8}  q^{12}  (q^{2} - q + 1)^{4}} $.\\ &&&\\ \hline
\end{tabular}
\end{center}
\end{example}

\subsection{Modular Representation theory}
In~\cite{PM}, Pforte and Murray studied representations of the group $\mathbb Z/2\mathbb Z \times \mathbb Z/2\mathbb Z$ in characteristic $2$. Notice that this amounts to determining a pair of commuting order $2$ matrices in $GL_d(F)$ where $F$ is a field of characteristic $2$. Further, determining orthogonal and symplectic representations amounts to determining such pairs in orthogonal and symplectic groups. If we want to have non-equivalent representations, we need such pair of matrices up to simultaneous conjugacy. Thus, the problem of determining non-equivalent representations (orthogonal, symplectic) of $(\mathbb Z/ p^n\mathbb{Z} )$ is closely related to the topic of determining $n$-tuples of commuting elements of order $p$ up to simultaneous conjugacy in $GL_d(F)$ ($O_d(F), Sp_{2l}(F)$).

\section{Commuting tuples and topology}\label{CTT}

Let $G$ be a Lie group and $n$ a positive integer. The space $Hom(\mathbb Z^n, G)$ can be identified with the set $G^{(n)}$ of $n$-tuples of commuting matrices (described in the beginning of Section~\ref{CP}) in $G$ as follows: for $f\colon \Z^n \rightarrow G$ we associate $(g_1,\ldots, g_n) \in G^{(n)}$, where $g_i= f(0,\ldots, \underbrace{1}_{i^{th}}, \ldots, 0)$. The space $Hom(\mathbb Z^n, G)$ acquires topology from $G^n$. We refer to the thesis of Stafa~\cite{St} for detailed exposition on this. Now, $G$ acts on $Hom(\Z^n, G)$ by conjugation (which amounts to the simultaneous conjugation action of $G$ on $G^{(n)}$). We denote it by $\rep(\Z^n, G) = G\backslash Hom(\Z^n, G)$. Thus, calculation of size of $\rep(\Z^n, G)$ is the same as the enumeration of simultaneous conjugacy classes of $n$-tuples of commuting elements in $G$. The case when $G$ is a finite classical group has been discussed in Section~\ref{BR}. 

There are two broad questions considered here: (i) parametrising  $Hom(\mathbb Z^n, G)$ and $\rep(\mathbb Z^n, G)$, (ii) the topological nature of these spaces, such as, connected components, homology etc. We briefly describe the work done in this direction without going into the technical details, and refer an interested reader to various articles and references therein mentioned in this section.

The subject originated from the work of Edward Witten~\cite{Wi1, Wi2}, who in his study of supersymmetry and guage theory studied the set $Hom(\Z^n,G)$ for $n = 2$ and $3$. Borel, Friedmann and Morgan~\cite{BFM} studied ``almost'' commuting pairs and triples of a compact group $G$ (i.e., commuting modulo some subgroup of the centre). They also analysed conjugacy classes of triples of commuting elements of simple groups. In \cite{AC}, Adem and Cohen looked more generally at $Hom(\pi, G)$, where $\pi$ is a finitely generated abelian group and $G$ is a Lie group. For $\pi$, a free abelian group of rank $n$ (for example, $\mathbb Z^n$), they determined explicitly the cohomology of the spaces $Hom(\pi, G)$ for $\pi$ of rank $2$ and $3$. They looked at the orthogonal groups, the special orthogonal and special unitary groups. In~\cite{ACG}, Adem, Cohen and Gomez studied the space $Hom(\Z^n, G_{m,p})$, where $G_{m,p} = SU_p(F)^m/\Delta(\Z/p)$, a central product of $m$- copies of $SU_p(F)$, where $p$ is a prime. They determined the number of path-connected components of $Hom(\Z^n, G_{m,P})$. 

In~\cite{AG}, Adem and Gomez looked at certain Lie groups, namely finite products of the classical groups, $SU_d(F)$, $U_d(F)$, and $Sp_{2l}(F)$. They looked at $Hom(\pi, G)$ for $\pi$ a more general finitely generated abelian group, i.e., $\pi = \Z^n \oplus A$, where $A$ is a finite abelian group, and $G$ is a finite product of the earlier mentioned classical groups. 
In~\cite{TS}, Giese and Sjerve calculated the fundamental groups of the connected components of the homomorphism spaces $Hom(\Z^n,G)$. They determined these for, $SU_2$, $U_2$ (which are connected), and for the connected components of $Hom(\Z^n, G)$ in particular. The proved that the fundamental groups of these are $\Z^n$ in the case of $U_2$; $0$ for $SU_2$, and  $\Z_2^n$, and the quaternion ring $Q_8$ for the various connected components of $Hom(\Z^n, SO_3)$. In his thesis \cite{St}, Stafa dealt with spaces of commuting elements of Lie groups, along with polyhedral products. He focused on determining all the homologies of the homomorphism space $Hom(\Z^n,G)$ for a Lie group $G$ all at once. In 2014, Rojo~\cite{Ro} looked at $Hom(\Z^n, O_d)$, and did a precise enumeration of the number of connected components, and showed that $\rep(\Z^n,O_d)$ and $Hom(\Z^n,O_d)$ have the same number of connected components. He also did a precise calculation of the number of connected components of $Hom(\Z^n, GL_d(\R))$, and of $Hom(\Z^n, SO_d)$.

In \cite{PS}, Pettet and Souto looked at  a group $G$ of complex or real points in a reductive group, and showed that for any maximal compact subgroup $H$ of $G$, and positive integer $k$, there is a weak retraction from $Hom(\Z^n, G)$ to $Hom(\Z^n,H)$. They necessary sufficient conditions on $n$ and the connectedness of a group $G$ of complex points of a semi-simple algebraic group, for the space $Hom(\Z^n, G)$ to be  connected. They also looked at the fundamental group of the connected component in $Hom(\Z^n, G)$ of the trivial representation, for a reductive group $G$.


\end{document}